\documentclass[11pt]{amsart}

\usepackage{latexsym}
\usepackage{amssymb}
\usepackage{amsmath}
\usepackage{amsfonts}
\usepackage[mathscr]{eucal}
\usepackage{times}

\usepackage{amscd}

\usepackage{psfrag}

\usepackage{amsthm}
\usepackage[all,cmtip]{xy}
\usepackage{hyperref}
\usepackage{comment}

\input{xypic}

\usepackage{tikz}
\usetikzlibrary{backgrounds,calc}
\usetikzlibrary{calc,arrows}

\newtheorem{thm}{Theorem}[section]
\newtheorem{cor}[thm]{Corollary}
\newtheorem{lem}[thm]{Lemma}
\newtheorem{prop}[thm]{Proposition}

\theoremstyle{definition}

\newtheorem{rem}[thm]{Remark}

\newtheorem*{ack}{Acknowledgements}

\numberwithin{equation}{section}

\newcommand{\RR}{\mathbb{R}}
\newcommand{\HH}{\mathbb{H}}
\newcommand{\EE}{\mathbb{E}}

\newcommand{\Z}{\mathbb{Z}}

\newcommand{\E}{\mathbb{E}}

\newcommand{\nil}{\mathrm{Nil}}
\newcommand{\sol}{\mathrm{Sol}}
\newcommand{\SL}{\widetilde{\mathrm{SL}_2(\mathbb{R})}}

\newcommand{\action}{\curvearrowright}
\newcommand{\thin}{\mathrm{thin}}

\DeclareMathOperator{\curv}{curv}

\DeclareMathOperator{\Susp}{Susp}

\newcommand{\RP}{\mathbb{R}P}

\newcommand{\torus}{\mathrm{\mathbb{T}}}
\newcommand{\sphere}{\mathrm{\mathbb{S}}}

\newcommand{\pdosr}{\RP^2}

\newcommand{\tildeX}{\widetilde{X}}

\begin{document}

%-----------------------------------------------------------------------------
% BEGIN FRONT MATTER ----------------------------------------------------------
%-----------------------------------------------------------------------------

% TITLE

\title[On three-dimensional Alexandrov spaces]{On three-dimensional Alexandrov spaces}

% AUTHOR 1

\author[F. Galaz-Garcia]{Fernando Galaz-Garcia$^*$}
\address[Galaz-Garcia]{Mathematisches Institut, WWU  M\"unster, Germany.}
\email{f.galaz-garcia@uni-muenster.de}
\thanks{$^{*}$The author is part of  SFB 878: \emph{Groups, Geometry \& Actions}, at the University of M\"unster.}

% AUTHOR 2

\author[L. Guijarro]{Luis Guijarro$^{**}$}
\address[Guijarro]{ Department of Mathematics, Universidad Aut\'onoma de Madrid, and ICMAT CSIC-UAM-UCM-UC3M, Spain.}
\curraddr{}
\email{luis.guijarro@uam.es}
\thanks{$^{**}$ Supported by research grants MTM2008-02676,  MTM2011-22612 from the Ministerio de Ciencia e Innovaci\'on (MCINN) and MINECO: ICMAT Severo Ochoa project SEV-2011-0087}

% DATE
\date{\today}

% MATH SUBJECT CLASSIFICATION AND KEYWORDS


% ABSTRACT

\begin{abstract}
We study three-dimensional Alexandrov spaces with a lower curvature bound, focusing on extending three classical results on three-dimensional manifolds: First, we show that a closed three-dimensional  Alexandrov space of positive curvature, with at least one topological singularity, must be homeomorphic to the suspension of $\RP^2$;  we use this to classify, up to homeomorphism, closed, positively curved Alexandrov spaces of dimension three. Second, we classify closed three-dimensional Alexandrov spaces of nonnegative curvature. Third, we study the  well-known Poincar\'e Conjecture in dimension three, in the context of Alexandrov spaces, in the two forms it is usually formulated for manifolds. We first show that the only closed three-dimensional Alexandrov space that is also a homotopy sphere is the $3$-sphere; then we give examples of closed, geometric, simply connected three-dimensional Alexandrov spaces for five of the eight Thurston geometries, proving along the way the impossibility of getting such examples for the $\nil$, $\SL$ and $\sol$ geometries. We conclude the paper by proving the analogue of the geometrization conjecture for closed three-dimensional Alexandrov spaces. 
\end{abstract}
\maketitle

%-----------------------------------------------------------------------------
% END FRONT MATTER ----------------------------------------------------------
%-----------------------------------------------------------------------------

%-----------------------------------------------------------------------------
%			MAIN MATTER	
%-----------------------------------------------------------------------------

%---------------------------------------------
% SECTION: INTRODUCTION
%---------------------------------------------

\section{Introduction and results}
Alexandrov spaces (with a lower curvature bound) are a natural extension of  Riemannian manifolds and appear when looking at limits of the latter under the Gromov-Hausdorff distance or when taking quotients of Riemannian manifolds by isometric group actions. They provide the correct setting where to study many of the questions of global Riemannian geometry, and thus a lot of the efforts since their introduction have been directed towards extending to Alexandrov spaces what is known for Riemannian manifolds. This paper aims to do this for three-dimensional Alexandrov spaces. Three-dimensional manifolds have been extensively studied, and it seems reasonable to apply the considerable knowledge on such manifolds in the broader context of Alexandrov geometry.  

To facilitate the reading of this paper, we have included the basic results on Alexandrov geometry that a general reader needs to know at the end of this introduction; 
this should by no means serve as a substitute for the  usual references ~\cite{BBI,BGP}. A good reference for the basic results of $3$-manifold topology is Hempel's book ~\cite{He}.

%--------------------------------------------------------------------
% 1ST PART OF INTRO: POSITIVE CURVATURE
%--------------------------------------------------------------------

The motivation for our first result is the classification (up to diffeomorphism) of closed Riemannian $3$-manifolds with positive sectional curvature, given by Hamilton in \cite{Ha}. We obtain the corresponding statement  for Alexandrov spaces; the lack of differentiable structures in the Alexandrov setting means that we will get the classification up to homeomorphism.

% THEOREM: POSITIVELY CURVED CLASSIFICATION

\begin{thm}[Three-dimensional Alexandrov spaces of positive curvature]
\label{PC}
Let $X^3$ be a closed, positively curved three-dimensional Alexandrov space. If $X^3$ has a point with space of directions homeomorphic to $\RP^2$, then $X^3$ is homeomorphic to $\Susp(\RP^2)$, the suspension of $\RP^2$.
\end{thm} 

The proof of Theorem~\ref{PC}, along with some corollaries, is contained in Section~\ref{S:PC}. 
 
 % REM
 
 \begin{rem}
 We have been informed that Theorem~\ref{PC} has also been obtained, independently, in \cite{HaS}, but see also \cite{PeMO}. 
 \end{rem}

%--------------------------------------------------------------------
% 2ND PART OF INTRO: NONNEGATIVE CURVATURE
%--------------------------------------------------------------------

Some of the ideas in the proof of the above theorem can be pursued further to provide a complete description of closed, three-dimensional Alexandrov spaces of nonnegative curvature. In the theorem below, we distinguish between the flat and the non-flat cases; in the first case, we obtain a rigidity statement, while in the second case we can only give a topological description. We denote the the nonorientable $\sphere^2$-bundle over $\sphere^1$ by $\sphere^2\tilde{\times}\sphere^1$, and the suspension of $\RP^2$ by $\Susp(\RP^2)$.

% THEOREM: NONNEGATIVE CURVATURE CLASSIFICATION

\begin{thm}[Three-dimensional Alexandrov spaces of nonnegative curvature]
\label{T:NNC}
Let $X^3$ be a closed, nonnegatively curved three-dimensional Alexandrov space.
\smallskip
\begin{enumerate} 
\item If $X^3$ is a topological manifold, then one of the following holds: 
\medskip
\begin{itemize}
	\item $X^3$ is homeomorphic to a spherical space form, 
	\item $X^3$ is homeomorphic to $\sphere^2\times\sphere^1$, $\RP^2\times\sphere^1$, $\RP^3\#\RP^3$ or $\sphere^2\tilde{\times}\sphere^1$; or \smallskip
	\item $X^3$ is isometric to a closed, flat three-dimensional space form.  
\end{itemize}
\medskip
\item  If $X^3$ has a point with space of directions homeomorphic to $\RP^2$, then either:
\medskip
\begin{itemize}
	\item  $X^3$ is homeomorphic to $\Susp(\RP^2)$,  $\Susp(\RP^2)\#\Susp(\RP^2)$ or\smallskip
	\item $X^3$ is isometric to a quotient of a closed, orientable, flat three- dimensional manifold by an orientation reversing  isometric involution with only isolated fixed points.
\end{itemize}
\end{enumerate}
\end{thm}

All possible involutions with only isolated fixed points on  closed, orientable, flat three-dimensional space forms and their orbit spaces have been classified in the work of Kwun and Tollefson \cite{KwTo} and Luft and Sjerve \cite{LuSj}.

Theorems \ref{PC} and \ref{T:NNC} settle Conjectures 1.10 and 1.11 in \cite{MiYa}. There is also a smooth classification of compact Riemannian $3$-orbifolds with nonnegative curvature in Section 5.4 of \cite{KleLo}.

%-----------------------------------------------------------------------------------------------
% 2ND PART OF INTRO: PC VS GPC
%-----------------------------------------------------------------------------------------------

% PROP: GPC FOR 3-DIM ALEX SPACES

The second half of the paper studies some of the classical questions of $3$-manifolds in the broader context of Alexandrov spaces.  
First, we consider the Poincar\'e Conjecture applied to three-dimensional Alexandrov spaces instead of purely $3$-manifolds. We establish differences between three-dimensional Alexandrov spaces that are also homotopy spheres and three-dimensional Alexandrov spaces that are only simply connected. While the first class satisfies the same statement as for $3$-manifolds, there are plenty of simply connected three-dimensional Alexandrov spaces that are not homeomorphic to the $3$-sphere. 

\begin{prop}[Generalized Poincar\'e Conjecture for three-dimensional Alexandrov spaces]
\label{P:GPC}
A closed three-dimensional Alexandrov space that is also a homotopy sphere is homeomorphic to $\sphere^3$.
\end{prop}
We remark that the proof of  Proposition~\ref{P:GPC} also implies that a closed, simply connected three-dimensional Alexandrov space that is a homology sphere must be a topological manifold, and is therefore homeomorphic to the $3$-sphere.

For the following statement, we say that an Alexandrov space $X^3$ has a given Thurston geometry (cf.~\cite{Sc}) if $X^3$ can be written as a quotient of the corresponding geometry by some cocompact lattice. We will say that such an Alexandrov $3$-space  is \emph{geometric}.

\begin{thm}[Simply connected three-dimensional Alexandrov spaces]
\label{T:SCAS}
For each Thurston geometry, except for 
$\nil$, $\SL$ and $\sol$, there exist closed, geometric, simply connected three-dimensional Alexandrov spaces that are not homeomorphic to the $3$-sphere.
\end{thm}

We prove Proposition~\ref{P:GPC} and Theorem~\ref{T:SCAS} in Section~\ref{S:PCONJ}.

In the last section we consider the geometrization conjecture for three-dimensional Alexandrov spaces. 
We say that a closed three-dimensional Alexandrov space $X$ admits a geometric decomposition if there exists a collection of spheres, projective planes, tori and Klein bottles that decompose $X$ into geometric pieces.

% THM

\begin{thm}[Geometrization of three-dimensional Alexandrov spaces]
\label{T:GEO}
 A closed three-dimensional Alexandrov space admits a geometric decomposition into geometric three-dimensional Alexandrov spaces.  
\end{thm}

The proof of Theorem~\ref{T:GEO} can be found in Section~\ref{S:GEO}.

% SUBSECTION: OVERVIEW OF ALEXANDROV SPACES

\subsection*{A brief overview of three-dimensional Alexandrov spaces}
An \emph{Alexandrov space} is a locally compact metric space $X$ with an inner length metric that satisfies locally a lower curvature bound $k_0$ in the Alexandrov sense; roughly speaking, this means that geodesic triangles in $X$ are ``fatter'' than equivalent ones in the space form of constant curvature $k_0$. This condition has strong consequences on the structure of the space. For our purposes, we recall the following three: 

% LIST

\begin{enumerate}
\item We can define tangent directions at any point of the space, and an angle distance between them. The metric completion of the tangent directions yields the so-called \emph{space of directions of $X$ at $p$}, which is usually denoted by $\Sigma_p$. With the distance induced by the angle distance, $\Sigma_p$ is again an Alexandrov space  of dimension $\dim X-1$ and with curvature greater than or equal to one. 
\item Positively curved closed Alexandrov spaces in dimension two are necessarily homeomorphic to either $\sphere^2$ or to the real projective plane $\RP^2$. 
\item Any point $p$ in $X$ has a neighborhood homeomorphic to a topological cone over the space of directions at $p$; this is the content of Perelman's conical neighborhood theorem (see \cite{Pe}).
\end{enumerate}

 Because of the third statement,  an Alexandrov $3$-space $X$ (without boundary)  that is not homeomorphic to a topological $3$-manifold must have some point whose space of directions $\Sigma$ is not homeomorphic to a $2$-sphere; since $\Sigma$ must be positively curved, it must be homeomorphic to  the real projective plane $\RP^2$.  
 Hence, the conical neighborhood theorem implies that $X$ is homeomorphic to a  $3$-manifold with a finite number of $\RP^2$-boundary components where we glue in cones over $\RP^2$. However, we can improve this and exhibit $X$ as the base of a two-fold branched cover $\pi:Y\rightarrow X$ whose total space $Y$ is a closed orientable manifold and whose branching set is the set of points with space of directions homeomorphic to $\RP^2$.
 
 \begin{lem}
 \label{L:BR_COV}
 Let $X$ be a closed three-dimensional Alexandrov space. If $X$ is not a topological manifold,  then there is a closed, orientable $3$-manifold $Y$ and an orientation reversing involution $\iota:Y\to Y$ with isolated fixed points such that $X$ is homeomorphic to the quotient $Y/\iota$. 
 \end{lem} 
 
\begin{proof}
Let $x_1, \ldots, x_n$ be the points in $X$ whose space of directions is a projective plane; remove disjoint open conical neighborhoods around each $x_i$ to obtain a $3$-manifold $X_0$ with $n$ boundary components.  The manifold $X_0$ is  nonorientable, because each projective plane in the boundary is two-sided. Let $\tilde{X}_0\to X_0$ be the orientable double cover of $X_0$. Each boundary component in $\tilde{X}_0$ is now a $2$-sphere, and we can close $\tilde{X}_0$  by gluing $3$-balls to its boundary to obtain the closed orientable manifold $Y$. It is clear now that the involution on $\tilde{X}_0$ can be extended to an involution on $Y$ whose quotient is $X$.
\end{proof}

The following lemma is a consequence of recent work of Grove and Wilking \cite[Section~5]{GW} and will be assumed throughout our paper. 

\begin{lem}
 \label{L:BR_COV_GEOM}
 Let $X$ be a closed three-dimensional Alexandrov space with curvature bounded below by $k$, with $k=1$ or $0$,  and assume that $X$ is not a topological manifold. If $Y$ is the orientable double branched cover of $X$ in Lemma~\ref{L:BR_COV}, then the following hold:
 \begin{enumerate}
 	\item The metric in $X$ can be lifted to $Y$ so that  $Y$ is an Alexandrov  space with curvature bounded below by $k$.
	\item The involution $\iota:Y\to Y$ is an isometry.
 \end{enumerate}
 \end{lem} 

Finally, we point out that the involution $\iota: Y\rightarrow Y$ in Lemmas~\ref{L:BR_COV} and \ref{L:BR_COV_GEOM} is locally linear, as a consequence of the work of Hirsch, Smale \cite{HS} and Livesay \cite{Li}. By a result of Kwasik and Lee \cite{KwLe}, the involution $\iota$ on $Y$ is equivalent to a PL involution. This  will allow us to use results on topological $3$-manifolds  equipped with PL involutions, which were extensively studied in the 1970's and 1980's. 

 % ACKNOWLEDGEMENTS

\begin{ack}
 The authors would like to thank Jos\'e Carlos G\'omez La\-rra\~{n}aga, Karsten Grove, Wolfgang Heil, Jos\' e Mar\'{\i}a Montesinos, Joan Porti and Burkhard Wilking for helpful conversations. The authors would also like to thank the Posgrado de Excelencia Internacional en Matem\'aticas at the Universidad Aut\'o\-noma de Madrid, where part of the present work was completed.
\end{ack}

%-----------------------------------------------------------------------------------------------
% SECTION: PROOF OF THEOREM 2 (POSITIVELY CURVED CASE)
%-----------------------------------------------------------------------------------------------

\section{Alexandrov $3$-spaces of positive curvature}
\label{S:PC}

\subsection{Proof of Theorem~\ref{PC}} Let $X$ be a closed three-dimensional Alexandrov space with positive curvature. We may assume, after rescaling the metric if necessary, that $\curv X \geq 1$. 
 Let $X'$ be the set of points in $X$ whose space of directions is homeomorphic to $\RP^2$. By hypothesis, $X'$ is nonempty. Recall that each element in $X'$ has a neighborhood homeomorphic to the Euclidean cone $C_0(\RP^2)$. By compactness, the set $X'$ is finite.

Let $x_1,\ldots,x_k$ be the points in $X'$. After removing  a neighborhood homeomorphic to $C_0(\RP^2)$ of each $x_i$, we obtain a topological $3$-manifold $X_0$ with boundary $k$ copies of  $\RP^2$. It is easy to see that $k$ is an even number, although we will not need that in what follows.

Let $\pi:Y\rightarrow X$ be the two-fold branched cover over $X$ with branching set $X'$, as in Lemma~\ref{L:BR_COV}. Let $y_i=\pi^{-1}(x_i)$, $i=1,\ldots,k$ and let $Y'=\{\,y_1,\ldots,y_r\,\}$. 
By Lemma~\ref{L:BR_COV_GEOM}, $Y$ is an Alexandrov space with $\curv\geq 1$. Since $Y$ has positive curvature bounded away from zero, it has finite fundamental group. On the other hand, $\pi_1(Y)\simeq \pi_1(Y\setminus Y')$, since $Y'$ is a finite set of points in $Y$. Since $\pi:Y\setminus Y'\rightarrow X\setminus X'$ is a regular two-fold cover, $\pi_*(\pi_1(Y\setminus Y'))$ is a subgroup of index $2$ in $\pi_1(X\setminus X')$. Hence, $\pi_1(X\setminus X')$ is finite. It follows from Epstein's theorem (cf. \cite[Chapter 9]{He}), and Perelman's proof of the Poincar\'e Conjecture, that $X\setminus X'$ is homeomorphic to $\RP^2\times [0,1]$. Thus $k=2$ and the conclusion of the theorem follows. Observe that $Y$ is homeomorphic to $\sphere^3$ and, by work of Hirsch, Smale \cite{HS} and Livesay \cite{Li}, the action of $\mathbb{Z}_2$ corresponding to the two-fold branched cover is equivalent to a linear action.   \hfill $\square$

\subsection{Corollaries.} We now list some consequences of Theorem~\ref{PC}. 
 
  % COROLLARY
 
\begin{cor}
\label{C:SC3DCLAS}
A closed, simply connected three-dimensional Alexandrov space of positive curvature is homeomorphic to $\sphere^3$ or to $\Susp(\RP^2)$.
\end{cor} 

 % COROLLARY

\begin{cor}
\label{C:3DCLAS}
A closed, three-dimensional Alexandrov space of positive curvature is homeomorphic to a spherical $3$-manifold or to $\Susp(\RP^2)$.
\end{cor} 

 % COROLLARY

\begin{cor}
\label{C:4DSDIR} The space of directions of a $4$-dimensional Alexandrov space without boundary is homeomorphic to $\Susp(\RP^2)$ or to a spherical $3$-manifold. 
\end{cor}

% COROLLARY

\begin{cor}
\label{C:SUSP} A closed $4$-dimensional Alexandrov space of  curvature bounded below by $1$ and diameter greater than $\pi/2$ is homeomorphic to the suspension of a spherical $3$-manifold or to $\Susp^2(\RP^2)$, the double suspension of $\RP^2$. 
\end{cor}

Corollaries~\ref{C:SC3DCLAS} and~\ref{C:3DCLAS} follow from Perelman's proofs of the Poincar\'e Conjecture and Thurston's Elliptization Conjecture, along with the fact that an Alexandrov space of positive curvature has finite fundamental group. Corollary~\ref{C:4DSDIR} is a consequence of Corollary~\ref{C:3DCLAS}, since the space of directions at any point of an $n$-dimensional Alexandrov space is isometric to a compact $(n-1)$-dimensional Alexandrov space with curvature bounded below by $1$. Corollary~\ref{C:SUSP} follows from the fact that an $n$-dimensional Alexandrov space of curvature bounded below by $1$ and diameter greater than $\pi/2$ is homeomorphic to the suspension of a compact $(n-1)$-dimensional Alexandrov space of curvature bounded below by $1$.

Recall that an Alexandrov space is an \emph{Alexandrov manifold} if it is homeomorphic to a topological manifold, and an Alexandrov space is called \emph{topologically regular} if every space of directions is homeomorphic to a sphere. Clearly, a topologically regular Alexandrov space is an Alexandrov manifold, but the converse is not necessarily true. Indeed, recall that the double  suspension $\Susp^2(P)$ of the Poincar\'e homology sphere $P$ is homeomorphic to $\sphere^5$. Since $P$ admits a Riemannian metric of constant positive curvature, $\Susp^2(P)$ admits an Alexandrov metric $d$ of positive curvature, given  by considering $\Susp^2(P)$ as a double spherical suspension. It follows that $(\Susp^2(P),d)$ is a five-dimensional Alexandrov manifold. On the other hand, $(\Susp^2(P),d)$ is  is not topologically regular, since it contains  points whose space of directions is homeomorphic to $\Susp(P)$, which is not a manifold. Using Theorem~\ref{PC} we recover the following
result, implicit in V.~Kapovitch's paper \cite{Ka}.

% COROLLARY

\begin{cor}
\label{C:TRALEX} Let $X^n$ be an $n$-dimensional Alexandrov manifold. If $n\leq 4$, then $X^n$ is topologically regular. 
\end{cor}
\begin{proof} If $n\leq 3$, the conclusion follows from the fact that every $1$- or $2$-dimensional Alexandrov space must be homeomorphic to a topological manifold. Suppose now that  $n=4$. Recall that any sufficiently small neighborhood $U$ of $p$ is homeomorphic to the cone over the space of directions $\Sigma_pX$ at $p$. Since a cone over a non-simply connected  $3$-manifold cannot be homeomorphic to the $4$-ball $D^4$, the only  case we need to consider is when $\Sigma_pX$ is homeomorphic to $\Susp(\RP^2)$. In this case, to conclude that $X$ cannot be a topological manifold, it suffices to verify that some homology group $H_k(U,U-p)$ is not isomorphic to $H_k(D^4,\sphere^3)$.  This follows easily from the long  exact sequence of the pair $(U,U-p)$.
\end{proof}

%---------------------------------------------------------
% SECTION: NONNEGATIVE CURVATURE
%---------------------------------------------------------

\section{Alexandrov $3$-spaces of nonnegative curvature}
\label{S:NNC}

\subsection*{Proof of Theorem~\ref{T:NNC}} Let $X$ be a three-dimensional Alexandrov space with nonnegative curvature. We will consider two possibilities, depending on whether $X$ is or not a topological manifold. 
Suppose first that $X$ is a topological manifold. We consider two subcases, depending on whether the fundamental group $\pi_1(X)$ is finite or not.

\begin{enumerate}
\item If $\pi_1(X)$ is a finite group, then the universal cover $\tildeX$ is a closed, simply connected topological manifold. Therefore, $\tildeX$ is homeomorphic to the $3$-sphere and $X$ will be homeomorphic to  a spherical $3$-manifold.

\item If $\pi_1(X)$ is infinite, the Splitting Theorem implies that $\tildeX$ is isometric to a product $\RR\times\tilde{Y}$, where $\tilde{Y}$ is a simply connected $2$-dimensional Alexandrov space with nonnegative curvature. The only possibilities for $\tilde{Y}$ are $\sphere^2$ or $\RR^2$. 

\begin{enumerate} 
\item If $\tilde{Y}$ is a topological $2$-sphere,  then $X$ is  covered by $\RR\times \sphere^2$, and consequently will be homeomorphic to either $\sphere^2\times\sphere^1$ or $\RP^3\#\RP^3$ (if orientable) or to $\sphere^1\times\RP^2$ or $\sphere^2\tilde{\times}\sphere^1$, the nonorientable $\sphere^2$-bundle over $\sphere^1$ (if nonorientable) (see \cite{To}).

\item If $\tilde{Y}$ is homeomorphic to $\RR^2$, then its metric has a compact quotient by isometries (since we are assuming that $X$ is closed). A second application of the Splitting Theorem yields that $\tilde{Y}$ must be isometric to Euclidean two-dimensional space $\mathbb{E}^2$. It follows that  $\tildeX$ isometric to $\E^3$. Hence $X$ must be isometric to one of the closed flat $3$-manifolds appearing in \cite{Wo}, pages 117 and 120 for the orientable and nonorientable cases, respectively.
\end{enumerate}
\end{enumerate}
Assume now that $X$ is an Alexandrov space with a finite number of topological singularities. As mentioned in the introduction, there  exists a double branched cover $\iota: M\to X$, where $M$ is an orientable topological $3$-manifold, and by results of \cite{GW}, $M$ with the induced metric is an Alexandrov space with nonnegative curvature and $\iota$ is an isometric orientation reversing involution whose fixed points are the branching points of the double branched cover, so that $X$ is isometric to $M/\iota$.

Therefore, we obtain all  possible spaces $X$ by considering orientation reversing isometric involutions with isolated fixed points on closed, orientable Alexandrov $3$-manifolds $M$ with nonnegative curvature; we have already determined the possibilities for $M$ in the first part of the proof.

\begin{enumerate}
\item The case where $M$ is elliptic has been covered in Theorem~\ref{PC}, thus obtaining that $X$ is homeomorphic to the suspension of $\RP^2$. 

\item When $M$ is a quotient of $\sphere^2\times\RR$, it must be homeomorphic to either $\sphere^2\times\sphere^1$ or $\RP^3\#\RP^3$. We have the following possibilities. 

% LIST

\begin{enumerate}
\item Involutions on $\sphere^2\times\sphere^1$ were considered in \cite{To2}.  There is only one with a finite fixed point set: It is the product of $C:\sphere^1\to\sphere^1$ given by reflection along a diameter, and of 
$A:\sphere^2\to\sphere^2$ with formula $A(x,y,z)=(x,-y,-z)$. It has $4$ fixed points; the quotient of $\sphere^2\times\sphere^1$ is homeomorphic to the union of two mapping cylinders of the quotient map $\sphere^2\to\RP^2$ or, equivalently, to $\Susp(\RP^2)\#\Susp(\RP^2)$.
\item There are four possible involutions on  $\RP^3\#\RP^3$ reverting orientation (see \cite{PKK}, page 472 for the right references). By work of Kim and Tollefson \cite{KT}, such an involution is either the obvious one, which exchanges the summands and fixes a $2$-sphere, or can be exhibited as a connected sum  $\iota_1\#\iota_2$ of involutions $\iota_i$ on each $\RP^3$ summand. This connected sum of involutions is taken along an appropriate ball intersecting a fixed point set component of $\iota_i$. By work of Kwun \cite{Kw}, there is exactly one orientation reversing involution on $\RP^3$, namely, the one induced by reflection along an equator of $\sphere^3$. Since this involution on $\RP^3$ has fixed point set a $2$-sphere and an isolated point, it follows that the fixed point set of any orientation reversing involution $\iota_1\#\iota_2$ on $\RP^3\#\RP^3$ has a two-dimensional component. Therefore, $\RP^3\#\RP^3$ cannot be the double branched cover $M$ of $X$. 
\end{enumerate}

\item Finally, when $M$ is flat, we are in the last situation mentioned in the theorem, and the statement follows because we had that $X$ is $M$ quotiented by an involution.

\end{enumerate}

Case (3) can be studied further: since the universal cover of $M$ is $\E^3$, it is clear that any isometric involution $a:M\to M$ will lift to an involution $A$ of $\E^3$. Also, if we assume that $a$ has an isolated fixed point, we can assume that the origin $0$ is in its fiber, and that $A$ fixes it. Thus, because the differential of $a$ at $p$ is minus the identity, we have that $A(u)=-u$ for every $u\in\E^3$. The fundamental group of $M$ has an explicit description (see \cite{Wo}, for instance), so the fundamental group of $X$ will be generated by $\pi_1(M)$ and $A$. 
It is possible to  describe $X$ topologically (see \cite{KwTo,LuSj}) but we will not do it here for concision's sake. \hfill $\square$

%-----------------------------------------------------
% SECTION: POINCARE CONJECTURE
%-----------------------------------------------------

\section{The Poincar\'e Conjectures for Alexandrov $3$-spaces}
\label{S:PCONJ}

The usual three-dimensional Poincar\'e Conjecture asserts that a closed, simply connected three-dimensional manifold must be homeomorphic to the $3$-sphere. This is equivalent to the statement that a homotopy $3$-sphere must be homeomorphic to the $3$-sphere. However, this equivalence is no longer true for Alexandrov $3$-spaces, due to the lack of Poincar\'e duality in the presence of topological singularities. Therefore, in dimension three, the Poincar\'e Conjecture and the Generalized Poincar\'e Conjecture for Alexandrov spaces are no longer necessarily equivalent, as opposed to the manifold case. 

In this section we prove the Generalized Poincar\'e Conjecture for compact Alexandrov $3$-spaces (cf. Proposition~\ref{P:GPC}). We also provide examples of geometric compact simply connected Alexandrov $3$-spaces which are not homeomorphic to the $3$-sphere. These spaces furnish counterexamples to the Poincar\'e Conjecture for compact Alexandrov $3$-spaces in five of the eight Thurston geometries. We also show that such counterexamples do not exist in the remaining three geometries.

\subsection*{Proof of Proposition~\ref{P:GPC}} Suppose that $X$ has topological singularities, otherwise $X$ is a topological $3$-manifold and the result follows from Perelman's solution to the Poincar\'e Conjecture. Applying the Mayer-Vietoris sequence to the decomposition $X=X_0\cup \cup_{i=1}^{2k}C_0(\RP^2)$, we obtain that $H_3(X,\Z)=0$, which contradicts the assumption that $X$ is a homotopy $3$-sphere.\hfill $\square$ 

\subsection*{Proof of Theorem~\ref{T:SCAS}} We now provide examples of geometric compact simply connected Alexandrov $3$-spaces which are not homeomorphic to the $3$-sphere, for five of the eight Thurston geometries. These spaces arise as quotients of compact geometric $3$-manifolds by the action of an orientation reversing isometric involution with isolated fixed points. We also show that such spaces cannot exist for the geometries $\sol$, $\nil$ and $\sol$ by reasons we will explain below. 

\subsubsection*{$\sphere^3$} The quotient of the spherical suspension of the antipodal map on the round $2$-sphere of radius one yields an isometric involution on the round $3$-sphere with quotient isometric to $\Susp(\RP^2)$,  the spherical suspension of a round $\RP^2$.  

\subsubsection*{$\EE^3$} Let $T^3$ be a flat torus and let $\iota:T^3\rightarrow T^3$ be given by complex conjugation on each $\sphere^1$ factor. This involution has eight isolated fixed points and its quotient space $T^3/\iota$ is flat away from them. 
To see that $T^3/\iota$ is simply connected, write it as the union of two copies of the mapping cylinder of the quotient map $T^2\to \sphere^2$ induced by the involution determined by complex conjugation on each circle factor of $T^2$. After removing a small conical neighborhood around each one of the eight topologically singular points in the quotient space $T^3/\iota$, we obtain a non-orientable $3$-manifold with eight $\RP^2$ boundary components; this space was named \emph{the octopod} in \cite{LAH}.

\subsubsection*{$\HH^3$} This example is an application of the following two results of Panov and Petrunin: 

\begin{thm}[Theorem 1.4 \cite{PP}]
\label{T:PP}Given a finitely presented group $G$ there is a finite index subgroup $\Gamma' \subset \Gamma_{12}$ such that the fundamental group of $\HH^3/\Gamma'$ is isomorphic to $G$. Moreover, the subgroup $\Gamma' \subset \Gamma_{12}$ can be chosen so that the quotient space $\HH^3/\Gamma'$ is a pseudomanifold with no boundary. In other words, the singular points of $\HH^3/\Gamma'$ are modeled on the orientation preserving actions of $\Z_2$ and $\Z_2 \oplus \Z_2$, and on the action of $\Z_2$ by central symmetry.
\end{thm}

\begin{cor}[Corollary 1.5 \cite{PP}] Any finitely presented group $G$ is isomorphic to the fundamental group of $M/\Z_2$, where $M$ is a closed oriented three-dimensional manifold and the action of $\Z_2$ on  $M$ has only isolated fixed points.
\end{cor}

Taking $G$ as the trivial group yields the desired example. 

\subsubsection*{$\sphere^2\times\RR$} There is an involution on $\sphere^2\times\sphere^1$ that acts as the antipodal map on $\sphere^2$ and as conjugation on $\sphere^1$; it is well known that its quotient space is homeomorphic to $\Susp(\RP^2)\#\Susp(\RP^2)$.

\subsubsection*{$\HH^2\times\RR$} The example given for $\EE^3$ can be adapted to this geometry; to do this, recall that a \emph{hyperelliptic involution} of a compact orientable surface $\Sigma$ of genus $g$ is an involution $h:\Sigma\to\Sigma$ whose 
quotient is the 2-sphere.  When $\Sigma$ receives a hyperbolic metric, $h$ can be realized as an isometry. As in the euclidean case,  we can construct the quotient of 
 $\Sigma\times\sphere^1$ by the involution $(h,\tau)$, where $\tau$ acts by conjugation on $\sphere^1$. As previously said, the quotient space can be written as the union of two copies of the mapping cylinder of $h:\Sigma\to \sphere^2$, thus being simply connected.
 
\subsubsection*{$\SL$ and $\nil$}  It is known that  these spaces are \emph{chiral}, i.e.~any isometric involution of  $\SL$ and $\nil$ preserves the orientation (see \cite{Bo}, section 2.4, for instance). If a simply connected three-dimensional Alexandrov space $X$ admits any of these geometries and is not a topological $3$-sphere, it would need to have a discrete set of points whose space of directions are $\pdosr$. The double branched cover of $X$ would be a manifold $M^3$ with  $\SL$ or $\nil$ geometry and an isometric involution $\tau$ such that $M^3/\tau\simeq X$. The fixed points of $\tau$ are mapped to the topologically singular points of $X$ under the quotient map, and therefore $\tau$ would be forced to revert the orientation. By lifting $\tau$ to an isometry of $\SL$ or $\nil$ we would get a contradiction to chirality.

\subsubsection*{$\sol$} As already stated, we cannot put this geometry on a closed, simply connected Alexandrov $3$-space with topological singularities; however the proof differs from the one for $\SL$ and $\nil$ since, in contrast to those cases, $\sol$ does admit isometric involutions reverting its orientation. In fact, we will show directly that there are no orientation reversing involutions with fixed points on compact $\sol$-manifolds. 

Recall that a closed $\sol$ $3$-manifold must be homeomorphic to one of the following:

\smallskip
\begin{enumerate}
\item A torus bundle over $\sphere^1$ with a hyperbolic gluing map, or
\item A union of two twisted $I$-bundles over the Klein bottle  glued by some map $\Phi:\torus\to\torus$ of the boundaries; these spaces are usually known  in the literature as \emph{sapphires} (cf.~\cite{SWW}).
\end{enumerate}

The first case is dealt with using Corollary~$1$, page 102 in \cite{KS}: no such bundles admit orientation reversing involutions with fixed points and thus cannot produce three-dimensional Alexandrov spaces that are not topological manifolds. 

The second case is a little more involved. Denote by $N$ such a sapphire and by $f:N\to N$ an isometric involution that reverts the orientation and with nonempty fixed point set $F$. There is a double cover $\pi:M\to N$, where $M$ is a torus bundle over $\sphere^1$ as before (see for instance Proposition~3.2 in \cite{GoWo}). If we denote by $L$ the subgroup of $\pi_1(N)$ corresponding to $M$, Theorem~3.3 in \cite{GoWo} asserts that $L$ is invariant by any endomorphism of $\pi_1(N)$. This implies that $f:N\to N$ lifts to an orientation reversing isometric involution $\bar{f}:M\to M$. If $p\in F$ is a fixed point of $f$, then $\bar{f}$ interchanges the points in $\pi^{-1}(p)$, since $\bar{f}$ does not admit fixed points. Since $\bar{f}$ have to preserve the fibers of $M$ over $S^1$ (this follows, for instance, from the arguments given in Theorem~8.2 in \cite{MS}), $\bar{f}$ would have to leave the fiber through $\pi^{-1}(p)$ invariant. However, according to Theorem~B in \cite{KS}, $\bar{f}$ is conjugate to an 
involution of $M$ that moves away any fiber from itself, thus giving the desired contradiction. \hfill $\square$

\begin{rem}
%Remarks on the topology of simply connected $3$-Alexandrov spaces
The complete topological description of closed, simply connected three-dimensional Alexandrov spaces seems out of reach at this time. A tempting conjecture would have been to assert that any such space is the connected sum of $\sphere^3$ and suspensions over projective planes. However, the following argument (indicated to us by B. Wilking) shows this not to be the case in general. The spherical metric in $\Susp(\RP^2)$ has positive scalar curvature away from the singular points; by a result of Gromov and Lawson \cite{GL} the connected sum of any number of such suspensions admits a metric with positive scalar curvature that remains an orbifold metric, since it does not differ from the spherical metric around the singular points. Given a homeomorphism between the above connected sum and the previous flat example, we would get a homeomorphism between their oriented branched coverings, and with the pullback metrics we would obtain a positive scalar curvature metric on the $3$-torus, contradicting results in \cite{SY}.
\end{rem}

%-----------------------------------------------------
% SECTION: GEOMETRIZATION
%-----------------------------------------------------

\section{Geometrization of three-dimensional Alexandrov spaces}
\label{S:GEO}

\begin{proof}[Proof of Theorem \ref{T:GEO}]
Let $X$ be a closed three-dimensional Alexandrov space. We assume that $X$ is not a topological manifold; otherwise, the theorem follows from Perelman's proof of Thurston's Geometrization Conjecture. Let $M$ be the orientable double branched cover of $X$. Recall that $M$ is a closed orientable topological manifold equipped with an  involution $\iota: M\rightarrow M$ such that $X=M/\iota$. As remarked after Lemma~\ref{L:BR_COV_GEOM}, the involution is locally linear and, by \cite[Corollary~2.2]{KwLe}, the map $\iota$ is equivalent to a smooth involution on $M$ considered as a smooth $3$-manifold. We may therefore consider $X$ as a smooth closed $3$-orbifold. Moreover, we may equip $M$ with an invariant Riemannian metric so that $\iota$ acts isometrically. Therefore, $M$ is a closed Riemannian manifold with an isometric action of $\Z_2$

By the work of Dinkelbach and Leeb \cite[Section 5]{DL}, there exists an equivariant Ricci flow $\Z_2\action \mathcal{M}$ with surgery   on $M$.  It is known (cf. \cite[Section 67]{KL}) that each  connected component $N_i$ of the time slab $M_{k-1}$ that goes extinct at the singular time $t_k$ is diffeomorphic to a spherical space form, to $\RP^3\#\RP^3$ or to $\sphere^2\times \sphere^1$. There are two possibilities: either $N_i$ is invariant under the $\Z_2$ action, or the $\Z_2$ action sends $N_i$ to a different connected component $N_i'$. In the first case, by \cite[Corollary~4.5]{Di}, $N_i$ is an equivariant connected sum of standard actions on components diffeomorphic to spherical space forms, $\sphere^2 \times \sphere^1$ and $\RP^3\#\RP^3$. In the second case, $N_i$ maps to a geometric component of the Alexandrov space $X$.

 After a sufficiently long time, so that every component that goes extinct in finite time has disappeared, we obtain a thick-thin decomposition of $M_k$, for $k$ large enough (see, for example, \cite{BBMBP}). Since the definition of the thick-thin decomposition depends entirely on the metric, the decomposition is preserved by the $\Z_2$ action.
 
The thick part flows in the limit to a hyperbolic metric and the $\Z_2$-action on such manifold is standard by \cite[Theorem~H]{DL}. On the other hand, $M_\thin$ is a graph manifold; Waldhausen proved that such manifolds can be written as the connected sums of submanifolds $N_i$ such that the Jaco-Shalen-Johannson decomposition of $N_i$ results in Seifert pieces, possibly with toroidal boundary components.  By work of Meeks and Yau \cite{MY} (see also \cite{JR}), the connected sum decomposition can be taken to be invariant under the $\Z_2$ action. Recall that each Seifert manifold is geometric (see \cite[Section 4]{Sc}), and, if left invariant by $\Z_2$, the action is standard. We distinguish two cases. 
 \\
 
 \noindent\textbf{Case 1.} Suppose that some $N_i$ has a non-trivial JSJ-decomposition. Then $N_i$ is Haken, since the JSJ tori are incompressible. If $N_i$ is not a torus bundle over a circle, it follows from a result of Meeks and Scott \cite{MS} that its JSJ-decomposition can be done equivariantly with respect to the $\Z_2$-action; therefore the Seifert decomposition is preserved by the $\Z_2$ action. If $N_i$ is a torus bundle over a circle, then it already admits a geometric structure, even though the standard JSJ decomposition usually cuts it along a torus \cite{Ne}.
 \\
 
 \noindent\textbf{Case 2.} Every $N_i$ has a trivial JSJ-decomposition. Then each $N_i$ is a Seifert manifold and $M_\thin$ is a connected sum of Seifert manifolds and each $N_i$ is geometric.
 
\end{proof}
% ----------------------------------------------------------------
% ----------------------------------------------------------------
% BEGIN BIBLIOGRAPHY----------------------------------------------
% ----------------------------------------------------------------

\bibliographystyle{amsplain}

% ----------------------------------------------------------------
% END BIBLIOGRAPHY-------------------------------------------------
% ----------------------------------------------------------------

\end{document}